\documentclass[oneside,english]{amsart}
\usepackage[T1]{fontenc}
\usepackage[latin9]{inputenc}
\usepackage{geometry}
\usepackage{amsthm}
\usepackage{amssymb}
\usepackage{esint}

\makeatletter
\numberwithin{equation}{section}
\numberwithin{figure}{section}
\theoremstyle{plain}
\newtheorem{thm}{\protect\theoremname}
  \theoremstyle{plain}
  \newtheorem{lem}[thm]{\protect\lemmaname}

\newtheorem{constr}[thm]{\protect\constname}  
\providecommand{\constname}{Construction}

\makeatother

\usepackage{babel}
  \providecommand{\lemmaname}{Lemma}
\providecommand{\theoremname}{Theorem}

\begin{document}

\title{A Short Proof that Minimal Sets of Planar Ordinary Differential Equations
are Trivial}

\author{Ido Bright}

\address{Department of Applied Mathematics, University of Washington.}
\begin{abstract}
We present a short proof, relaying on the divergence theorem, verifying
that minimal sets in the plane are trivial.
\end{abstract}
\maketitle
\global\long\def\seqf#1{#1_{1}\left(\cdot\right),#1_{2}\left(\cdot\right),\dots}

\global\long\def\seq#1{#1_{1},#1_{2},\dots}

\global\long\def\Bz#1{B\left(\mathbf{x}^{*}\left(0\right),#1\right)}

\global\long\def\seqfn#1#2{#1_{1}\left(\cdot\right),#1_{2}\left(\cdot\right),\dots,#1_{#2}\left(\cdot\right)}

\global\long\def\pd#1#2{#1_{1}\left(\cdot\right),#1_{2}\left(\cdot\right),\dots,#1_{#2}\left(\cdot\right)}

\global\long\def\pd{\left(\cdot\right)}

\global\long\def\R{\mathbb{R}}

\global\long\def\Rt{\mathbb{R}^{2}}

\global\long\def\semi{[0,\infty)}

\global\long\def\x{\mathbf{x}^{*}\left(\cdot\right)}

\global\long\def\xs#1{\mathbf{x}^{*}\left(#1\right)}

\global\long\def\ys#1{\mathbf{y}^{*}\left(#1\right)}

\global\long\def\n#1{\mathbf{N}_{#1}}

\global\long\def\t#1{\mathbf{T}_{#1}}

\global\long\def\xd{\frac{d}{dt}\mathbf{x}^{*}\left(\cdot\right)}

\global\long\def\xdt{\frac{d}{dt}\mathbf{x}^{*}\left(t\right)}

\global\long\def\y{y^{*}\left(\cdot\right)}

\global\long\def\y{\mathbf{y}\left(\cdot\right)}

\global\long\def\z{z^{*}\left(\cdot\right)}
`

\global\long\def\xt{\mathbf{x}^{*}\left(t\right)}

\global\long\def\yt{y^{*}\left(t\right)}

\global\long\def\zt{z^{*}\left(t\right)}

\section{Introduction}

We consider the ordinary differential equation in the plane defined
by 
\begin{equation}
\frac{d\mathbf{x}}{dt}=\mathbf{f}(\mathbf{x}),\label{eq:ode}
\end{equation}
where $\mathbf{f}:\R^{2}\rightarrow\R^{2}$ is a locally Lipschitz
function. (Although, we only use uniqueness with respect to initial
conditions of (\ref{eq:ode}) and the continuity of $\mathbf{f}$.).

A minimal set is a nonempty closed invariant set, which is minimal
with respect inclusions. A trivial minimal set is a set that is the
image if either a stationary solution or a periodic solution. 

We present a new short proof of the following well known result.
\begin{thm}
Any minimal set of (\ref{eq:ode}) either corresponds to a stationary
solution or to the image of a periodic solution, namely, all minimal
sets are trivial.
\end{thm}
The text-book proof of this theorem relays on the Poincaré\textendash{}Bendixson
theorem, and employs dynamical arguments. The proof presented in this
paper relays on a different argument, relaying on a property of the
velocity of Jordan curves. This idea was introduced in \cite{art-bright}
and further developed in \cite{Bright201284,Bright_Lee}.

\section{Proof of Main Result}

In the proof of the main result we use the following notation. The
2-dimensional euclidean space is denoted by $\R^{2}$, and the norm
of a vector $\mathbf{y}\in\R^{2}$ is denoted by $\left|\mathbf{y}\right|$.
The open ball in $\Rt$, centered at $\mathbf{y}$ with radius $r$,
is denoted by $B\left(\mathbf{y},r\right)$. The closure of an open
set \textbf{$O\subset\R^{2}$} is denoted by $\bar{O}$, its boundary
by $\partial O$, and its exterior normal and tangent vector at the
point $\mathbf{y}\in\partial O$ are denoted by $\n{\partial O}\left(\mathbf{y}\right)$
and $\t{\partial O}\left(\mathbf{y}\right)$, respectively.

We shall use the following results that are well known in the smooth
case.
\begin{lem}
\label{lem:zero int}Let \textbf{$O\subset\R^{2}$} be a bounded open
set with rectifiable boundary. Then 
\[
\mathbf{v}=\int_{\partial O}\n{\partial O}\left(\mathbf{y}\right)d\mathbf{y}=\mathbf{0}\in\Rt.
\]
\end{lem}
\begin{proof}
Assume in contradiction that $\mathbf{v}\neq\mathbf{0}$, and set
$\mathbf{g}:\R^{2}\rightarrow\R^{2}$ by $\mathbf{g}\equiv\frac{\mathbf{v}}{\left|\mathbf{v}\right|}$.
The divergence theorem for sets of finite perimeter (see, e.g., \cite{pfeffer2012divergence})
implies that 
\[
\left|\mathbf{v}\right|=\frac{\mathbf{v}}{\left|\mathbf{v}\right|}\mathbf{v}=\left|\int_{\partial O}\mathbf{g}\left(\mathbf{y}\right)\n{\partial O}\left(\mathbf{y}\right)d\mathbf{y}\right|=\left|\int_{O}\nabla\cdot\mathbf{g}\left(\mathbf{y}\right)d\mathbf{y}\right|=0,
\]
in contradiction.
\end{proof}
The following lemma appears in \cite{wells1975note}.
\begin{lem}
\label{lem:Sard}Suppose $I\subset\R$ is a bounded interval and $g:I\rightarrow\R$
is a Lipschitz function. Then for almost every $r\in\R$ the set $g^{-1}\left(r\right)=\left\{ t\in I|g\left(t\right)=r\right\} $
is finite.
\end{lem}
To prove the main theorem, let us now fix a minimal set $\Omega\subset\R^{2}$
 and a solution $\x$ of (\ref{eq:ode}), defined on $[0,\infty)$,
with trajectory contained in $\Omega$. 

We shall also use the following well known fact.
\begin{lem}
\label{lem:recurense}For every $\mathbf{y_{0}}\in\Omega$ and $\delta>0$
there exists $t>s$ such that $\left|\xs t-\mathbf{y_{0}}\right|<\delta$.\end{lem}
\begin{proof}
Otherwise, suppose that the lamma does not hold for some $\mathbf{y_{0}},\delta$
and $s$. Then the curve $\ys t=\xs{s+t}$ is a solution of (\ref{eq:ode})
with trajectory contained in $\Omega\backslash B\left(\mathbf{y_{0}},\delta\right)$
for a suitable $\delta>0$, in contradiction to the minimality of
$\Omega$. 
\end{proof}
If $\Omega$ is not a singleton we choose $D>0$ such that $\Omega\backslash B\left(\xs 0,3D\right)\neq\emptyset$
and apply the following construction:

\begin{constr}

\label{con:Construction} Set $\delta_{0}=D$ and $t_{0}$ as the
first time point where $\x$ meets $\partial\Bz{\delta_{0}}$. For
$i=1,2,\dots$ do the following:
\begin{enumerate}
\item Choose $\delta_{i}<\delta_{i-1}/2$ small enough, such that $\left|\xs 0-\xs t\right|>\delta_{i}$
for all $t\in\left[t_{0},t_{i-1}\right]$.
\item Set $t_{i}$ as the first time point after $t_{0}$ where the curve
$\x$ meets $\partial\Bz{\delta_{i}}$. (Here we use Lemma \ref{lem:recurense}).
\item Starting from $\xs{t_{i}}$ follow the line connecting it to $ $$\xs 0$,
until first meeting a point in $\xs{\left[0,t_{0}\right]}$. Let $\xs{s_{i}}$
be this point.
\item Let $\gamma_{i}$ be the parametrized Jordan curve obtained by following
the curve $\x$ in the interval $\left[s_{i},t_{i}\right]$ and then
the line connecting its endpoints, with velocity of norm $1$.
\end{enumerate}
\end{constr}
\begin{lem}
\label{lem:periodic so}If $t_{i}\rightarrow t^{*}$ then $\xs 0=\xs{t^{*}}$,
and $\x$ is a periodic solution with image $\Omega$.\end{lem}
\begin{proof}
According to our construction $\left|\xs 0-\xs{t_{i}}\right|=\delta_{i}<2^{-i}D$
for every $i$. Hence, by continuity $\xs{t^{*}}=\xs 0$, and $\x$
is periodic. By the minimality of $\Omega$, the image of $\x$ is
$\Omega$.
\end{proof}

\begin{proof}
[Proof of Theorem 1] Clearly, $\Omega$ is a singleton if and only
if it contains a point $\mathbf{y}\in\Omega$ such that $\mathbf{f}\left(\mathbf{y}\right)=\mathbf{0}$.
In this case and when the condition of Lemma \ref{lem:periodic so}
holds, we are done. Thus, we assume that $\mathbf{f}$ does not vanish
in $\Omega$ and that $t_{i}\rightarrow\infty$. 

Fix $\mathbf{y}_{0}\in\Omega$ such that $\left|\mathbf{y}_{0}-\xs 0\right|>2D$.
Using Lemma \ref{lem:Sard} we fix an arbitrary small ball $B=B\left(\mathbf{y}_{0},r_{0}\right)$,
such that $r_{0}<D$, and that $\left\{ 0\le t\le s|\left|\xt-\mathbf{y_{0}}\right|=r_{0}\right\} $
is finite for every $s>0$. Note that this implies that for every
$i$ the Jordan curve $\gamma_{i}$ intersects $\partial B$ at a
finite number of points, and that the portion of $\gamma_{i}$ in
$B$ corresponds to the trajectory $\x$. 

For every $i$ we denote the interior of $\gamma_{i}$ by $O_{i}$,
and, using the identity
\[
\partial\left(O_{i}\cap B\right)\subset\left(\partial O_{i}\cap B\right)\cup\left(O_{i}\cap\partial B\right)\cup\left(\partial O_{i}\cap\partial B\right),
\]
we obtain, by Lemma \ref{lem:zero int}, that 
\[
\mathbf{0}=\int_{\partial\left(O_{i}\cap B\right)}\n{\partial\left(O_{i}\cap B\right)}\left(\mathbf{y}\right)d\mathbf{y}=\int_{\partial O_{i}\cap B}\n{\partial O_{i}}\left(\mathbf{y}\right)d\mathbf{y}-\int_{O_{i}\cap\partial B}\n{\partial B}\left(\mathbf{y}\right)d\mathbf{y},
\]
since $\partial O_{i}\cap\partial B$ has zero measure. This bounds
\begin{equation}
\left|\int_{\partial O_{i}\cap B}\n{\partial O_{i}}\left(\mathbf{y}\right)d\mathbf{y}\right|=\left|\int_{O_{i}\cap\partial B}\n{\partial B}\left(\mathbf{y}\right)d\mathbf{y}\right|\le\left|\int_{O_{i}\cap\partial B}\left|\n{\partial B}\left(\mathbf{y}\right)\right|d\mathbf{y}\right|\le2\pi r_{0}.\label{eq:bound 2pir}
\end{equation}
For each $i$ the set $\partial O_{i}\cap B$ contains a finite number
of arcs, and applying a change of variable it is easy to see that
\[
\int_{\partial O_{i}\cap B}\n{\partial O_{i}}\left(\mathbf{y}\right)d\mathbf{y}=P_{i}\int_{\partial O_{i}\cap B}\t{\partial O_{i}}\left(\mathbf{y}\right)d\mathbf{y}=P_{i}\int_{\left\{ t\le t_{i}|\xs t\in B\right\} }\xdt dt,
\]
where $\t{\partial O_{i}}$ is chosen to agree with the direction
of $\gamma_{i}$, and $P_{i}$ is a $\frac{\pi}{2}$-rotation matrix.
Here we use the fact that the portion of $\gamma_{i}$ in $B$ corresponds
to the original trajectory $\x$. 

Combined with (\ref{eq:bound 2pir}) we conclude that for every $i$
\[
\left|\int_{\left\{ t\le t_{i}|\xs t\in B\right\} }\mathbf{f}\left(\xt\right)dt\right|\le2\pi r_{0}.
\]

The minimality of $\Omega$ and Lemma \ref{lem:recurense} implies
that the set $\left\{ t|\xs t\in B\right\} $ has infinite measure.
This implies that $\mathbf{0}$ is contained in the convex hull of
$\left\{ \mathbf{f}\left(\mathbf{y}\right)|\mathbf{y}\in\bar{B}\right\} $.
The radius $r_{0}$ can be chosen arbitrary small, thus, the continuity
of $\mathbf{f}$ implies that $\mathbf{f}\left(\mathbf{y}_{0}\right)=\mathbf{0}$,
in contradiction.

\bibliographystyle{plain}
\nocite{*}
\bibliography{bib}
\end{proof}

\end{document}